\documentclass[a4paper, 10pt, oneside, onecolumn]{article}

\RequirePackage{geometry}
\makeatletter
\g@addto@macro\bfseries{\boldmath}
\makeatother

\usepackage[utf8]{inputenc}
\usepackage[T1]{fontenc}
\usepackage{lmodern}
\usepackage{amsmath}
\usepackage{amssymb}
\usepackage{amsfonts}
\usepackage{amsthm}
\usepackage{mathtools}
\usepackage{xcolor}
\usepackage{hyperref}
\usepackage{cleveref}
\hypersetup{
	colorlinks=false,
	pdfborder={0 0 0},
	pdftitle={Global mass-preserving solutions for a two-dimensional chemotaxis system with rotational flux components coupled with a full Navier--Stokes equation},
	pdfauthor={Frederic Heihoff},
	pdfkeywords={chemotaxis; Navier--Stokes; global existence; generalized solutions},
	bookmarksopen=true,
}

\newtheorem{base}{Base}[section]
\numberwithin{equation}{section}

\theoremstyle{plain}
\newtheorem{theorem}[base]{Theorem}
\newtheorem{lemma}[base]{Lemma}

\theoremstyle{definition}
\newtheorem{definition}[base]{Definition}

\newtheorem{remark}[base]{Remark}

\newcommand{\R}{\mathbb{R}}
\newcommand{\loc}{\mathrm{loc}}
\newcommand{\N}{\mathbb{N}}
\renewcommand{\d}{\,\mathrm{d}}
\newcommand{\laplace}{\Delta}
\newcommand{\defs}{\coloneqq}
\newcommand{\stext}[1]{\;\;\text{ #1 }\;\;}
\newcommand{\grad}{\nabla}
\renewcommand{\div}{\nabla \cdot}
\renewcommand{\L}[1]{{L^{#1}(\Omega)}}
\newcommand\numberthis{\addtocounter{equation}{1}\tag{\theequation}}

\begin{document}
	
\title{Global mass-preserving solutions for a two-dimensional chemotaxis system with rotational flux components coupled with a full Navier--Stokes equation}
\author{
	Frederic Heihoff\footnote{fheihoff@math.uni-paderborn.de}\\
	{\small Institut f\"ur Mathematik, Universit\"at Paderborn,}\\
	{\small 33098 Paderborn, Germany}
}
\date{}
\maketitle

\begin{abstract}
	\noindent We study the chemotaxis--Navier--Stokes system
	\[
	\left\{\;\;
	\begin{aligned}
		n_t + u \cdot \nabla n &\;\;=\;\; \Delta n - \nabla \cdot (nS(x,n,c) \nabla c), \;\;\;\;\;\; &&x\in \Omega, t > 0, \\
		c_t + u\cdot \nabla c &\;\;=\;\; \Delta c - n f(c), && x\in \Omega, t > 0, \\ 
		u_t + (u\cdot \nabla) u  &\;\;=\;\; \Delta u + \nabla P + n \nabla \phi, \;\;\;\;\;\; \nabla \cdot u = 0, \;\;\;\;\;\; && x\in \Omega, t > 0
	\end{aligned}
	\right. \tag{$\star$}	
	\]
	with no-flux boundary conditions for $n$, $c$ and Dirichlet boundary conditions for $u$ in a bounded, convex domain $\Omega \subseteq \mathbb{R}^2$ with a smooth boundary, which is motivated by recent modeling approaches from biology for aerobic bacteria suspended in a sessile water drop. We further do not assume the chemotactic sensitivity $S$ to be scalar as is common, but to be able to attain values in $\mathbb{R}^{2\times2}$, which for example allows for more complex modeling of bacterial behavior near the boundary of the water drop. This is seen as a potential source of the structure formation observed in experiments.
	\\[0.5em]
	While there have been various results for scalar chemotactic sensitivities $S$ due to a convenient energy-type inequality and some for the non-scalar case with only a Stokes fluid equation (or other strong restrictions) simplifying the analysis of the third equation in ($\star$) significantly, we consider the full combined case under little restrictions for the system giving us very little to go on in terms of a priori estimates. We nonetheless manage to still achieve sufficient estimates using Trudinger--Moser type inequalities to extend the existence results seen in a recent work by Winkler for the Stokes case with non-scalar $S$ to the full Navier--Stokes case. Namely, we construct a similar global mass-preserving generalized solution for ($\star$) in planar convex domains for sufficiently smooth parameter functions $S$, $f$ and  $\phi$ and only under the fairly weak assumptions that $S$ is appropriately bounded, $f$ is non-negative and $f(0) = 0$. 
	\\[0.5em]	
\end{abstract}
	
\section{Introduction}
We will study a mathematical model from biology describing chemotaxis, the directed movement of cells along a chemical gradient towards an attractant, and its interaction with the fluid said cells are suspended in. While pure chemotaxis models, which generally consist of two coupled partial differential equations, without added fluid interaction have been quite heavily studied for some time now (initiated by the seminal work of Keller and Segel in \cite{keller1970initiation}), recent findings by Dombrowski et al.\ (cf.\ \cite{PhysRevLett.93.098103}) have shown that in populations of \emph{Bacillus subtilis} much stronger liquid movement than could originate from independent bacteria can be observed after cell aggregates have been formed. This calls into question the prior modeling assumption that the liquid-cell interaction can basically be disregarded because each single organism has negligible influence on the liquid. This observation then led Tuval et al.\ (cf.\ \cite{Tuval2277}) to an extended modeling approach, which adds a full Navier--Stokes type fluid equation to the standard chemotaxis model. The central new interactions are that the attractant and the cells suspended in the solution are affected by convective forces exerted by the fluid, while the fluid is affected by buoyant forces exerted by the cells because of their comparatively large size.
\\[0.5em]
In mathematical terms and after some normalization of parameters, this modeling approach by Tuval et al.\ leads to a system of coupled partial differential equations of the following form:
\[
\left\{\;\;
\begin{aligned}
	n_t + u \cdot \grad n &\;\;=\;\; \laplace n - \div (nS(x,n,c) \grad c),  &&x\in \Omega, t > 0, \\
 	c_t + u\cdot \grad c &\;\;=\;\; \laplace c - n f(c), && x\in \Omega, t > 0, \\ 
 	u_t + (u\cdot \grad) u  &\;\;=\;\; \laplace u + \grad P + n \grad \phi, \;\;\;\;\;\; \div u = 0, \;\;\;\;\;\; && x\in \Omega, t > 0.
\end{aligned}
\right.
\numberthis \label{problem}
\]
Here $\Omega \subseteq \R^N$ with $N \in \N$ is some space domain and the functions $n = n(x,t)$, $c = c(x,t)$, $u = u(x,t)$, $P = P(x,t)$ model the bacteria population, the attractant (e.g.\ oxygen) concentration and the fluid velocity field and associated pressure respectively. The system is further parameterized by the given functions $S$, which models the chemotactic sensitivity of the bacteria, $f$, which models the attractant consumption rate, and $\phi$, which represents the gravitational potential. Convection is modeled by the terms $u \cdot \grad n$ and $u \cdot \grad c$, while buoyant forces are modeled by the term $n \grad \phi$. 
\\[0.5em]
This system is fairly well understood if we assume the chemotactic sensitivity $S$ to be a scalar function as this brings the system quite close to the classical Keller--Segel model in terms of the chemotactic interaction (cf.\ \cite{keller1970initiation}). Amongst many other things, global well-posedness of this system in two or three dimensions with different assumptions on the initial data and parameter functions $f$, $S$, $\phi$ has been extensively studied (cf.\ e.g.\ \cite{MR3208807}, \cite{MR2754058}, \cite{WinklerExistence}, \cite{MR2838394}). In some two dimensional settings, there are global, unique classical existence results available (cf.\ \cite{WinklerExistence}).
While in three dimensional cases there are still existence results available, they are somewhat less ambitious and deal with weaker notions of solutions and only consider eventual smoothness properties (cf.\ \cite{WinklerExistence}, \cite{MR3542616}, \cite{MR3605965}). For a more broad survey of mathematical chemotaxis models and recent results about them, consult for instance \cite{MR3351175}.

\paragraph{Main result.} The central goal of this paper is to expand the weak existence results seen in \cite{WinklerStokesCase} to our system (\ref{problem}) containing a full Navier--Stokes fluid equation as opposed to the simpler Stokes equation in the reference. Or put more precisely, we will consider the following setting: 
\\[0.5em]
We analyze the system (\ref{problem}) in a convex, bounded domain $\Omega \subseteq \R^2$ with smooth boundary. We then add the boundary conditions
\begin{equation}
	\grad n \cdot \nu = n(S(x,n,c)\grad c)  \cdot \nu, \;\; \grad c \cdot \nu = 0, \;\; u = 0 \;\;\;\; \forall x \in \partial\Omega, t > 0 \label{boundary_conditions}
\end{equation}
and initial conditions
\begin{equation}
	n(x, 0) = n_0(x),\;\;c(x,0) = c_0(x),\;\;u(x,0) = u_0(x) \;\;\;\;  \forall x\in\Omega
	\label{initial_data}
\end{equation}
for initial data with the following properties:
\begin{equation}
	\left\{\;
		\begin{aligned}
		n_0 &\in C^{\iota}(\overline{\Omega})  && \text{ for some } \iota > 0\text{ with } n_0 \geq 0 \text{ and } n_0 \not\equiv 0 \text{ in } \Omega, \\
		c_0 &\in W^{1,\infty}(\Omega) &&\text{ with } c_0 \geq 0 \text{ in } \Omega,  \\
		u_0 &\in D(A_2^\vartheta)&& \text{ for some } \vartheta \in (\tfrac{1}{2}, 1)
		\end{aligned}
	\right.
	\label{initial_data_props}
\end{equation}
Here $A_2$ denotes the Stokes operator in the Hilbert space $L^2_\sigma(\Omega) \defs \{ \varphi \in (L^2(\Omega))^2 \;|\; \div \varphi = 0 \}$ of all solenoidal functions in $(L^2(\Omega))^2$. Further $W_{0,\sigma}^{k,2}(\Omega) \defs (W_{0}^{k,2}(\Omega))^2 \cap L^2_\sigma(\Omega)$ for $k \in \N$ denote the corresponding Sobolev spaces of solenoidal functions. 
\\[0.5em]
For the functions $f, S$ and $\phi$ that parametrize (\ref{problem}), we will throughout this paper assume that
\begin{equation}
	f \in C^1([0,\infty)) \stext{is nonnegative with } f(0) = 0
	\label{f_regularity}
\end{equation}
and that, for $S = (S_{ij})_{i,j \in\{1,2\}}$,
\begin{equation}
	S_{ij} \in C^2(\overline{\Omega}\times[0,\infty)\times[0,\infty)) \;\;\text{ for } i,j \in \{1,2\}
	\label{S_regularity}
\end{equation}
and that
\begin{equation}
	|S(x,n,c)| \leq S_0(c) \;\;\;\; \forall (x,n,c) \in \overline{\Omega}\times[0,\infty)^2 \;\; \text{ for some nondecreasing } S_0: [0,\infty) \rightarrow [0, \infty)
	\label{S_0_bound}
\end{equation}
and that
\begin{equation}
	\phi \in W^{2,\infty}(\Omega) \label{phi_regularity}.
\end{equation}
Under all of these assumptions, we then show that (\ref{problem}) has a global mass-preserving generalized solution:
\begin{theorem}
If we assume that $f$, $S$, $\phi$ satisfy (\ref{f_regularity})--(\ref{phi_regularity}), then the system (\ref{problem}) with boundary conditions (\ref{boundary_conditions}) and initial data (\ref{initial_data}) of regularity class (\ref{initial_data_props}) has a global mass-preserving generalized solution $(n,c,u)$ in the sense of \Cref{definition:weak_solution} below.	
\end{theorem}

\paragraph{Complications: Loss of energy structure and nonlinear convection.}
The above setting presents us with two key complications, one inherited from the Stokes case already discussed in \cite{WinklerStokesCase} and one we reintroduce from the original model by Tuval et al.
\\[0.5em]
The complication we inherited is of course the non-scalar chemotactic sensitivity $S$. Such $S$ are especially interesting from a biological point of view because scalar $S$ have been shown to lead to long time homogenization for some common parameters, which does not agree with the structure formation seen in experiments (cf.\ \cite{MR3149063}). As these observations further suggest that spatial inhomogeneities often originate at the boundary (cf.\ \cite{PhysRevLett.93.098103}), modern modeling approaches introduce rotational flux components near said boundary, which necessitate sensitivity functions that look somewhat like
\[
	S = a\left(\,\begin{matrix}1 & 0 \\ 0 & 1\end{matrix}\, \right) + b \left(\,\begin{matrix}0 & -1 \\ 1& 0\end{matrix} \,\right) \;\;\;\; \text{ for }a > 0, \; b \in \R
\]
with significant non-diagonal entries near the boundary (cf.\ \cite{MR3294964}, \cite{MR2505083}). The most devastating consequence of this for the analysis of the model is that this leads to the apparent loss of an energy type functional, which is present in the scalar case and often key to proving global existence of solutions (cf.\ \cite{MR2754058}, \cite{MR3369260}, \cite{WinklerExistence}, \cite{TemamNavierStokes}) or understanding their long time behavior (cf.\ \cite{MR3149063}, \cite{MR3605965}). Results dealing with non-scalar sensitivities therefore often analyze the system under some very strong restriction on either $S$ or the initial data (cf.\ \cite{MR3531759}, \cite{MR3401606}, \cite{MR3542964}).
\\[0.5em]
Even in the fluid-free version of (\ref{problem}), this loss 
of structure has led to a significant lack of knowledge and, e.g.\ for $N = 2$, global smooth solutions have thus far only been constructed under significant smallness conditions for $c_0$ (cf.\ \cite{MR3302296} or \cite{MR3562314} for an extension to the fluid case). For arbitrarily large data and dimension $N$, it is possible to construct global generalized solutions (not unlike those constructed in this paper) as seen in \cite{WinklerLargeDataGeneralized} at the very least.
\\[0.5em]
The complication we reintroduced is adding the nonlinear convection term, which is disregarded in \cite{WinklerStokesCase}, back into the third equation making it a full Navier--Stokes fluid equation. This means that especially the semigroup methods used in the Stokes case lose some of their effectiveness reducing our immediate access to regularity information for the third equation and making it in general much tougher to handle.

\paragraph{Main ideas.}
A lot of ideas from the Stokes case \cite{WinklerStokesCase} translate fairly immediately for the first two equations in (\ref{problem}), especially concerning the handling of the somewhat problematic chemotactic sensitivity $S$ as long as we still manage to provide similar bounds for $u$ as in the reference. Therefore it is in establishing these bounds where our key ideas come in. As already mentioned, semigroup techniques lose some of their fruitfulness in the third equation due to the newly introduced nonlinear convection term, but similarly to the Stokes case, the rather weak regularity information of the form
\[
	\int_0^\infty \int_\Omega \frac{|\grad n|^2}{(n+1)^2} \leq C
\]
in \Cref{lemma:basic_props} is of central importance yet again (In the Stokes case, it was mostly used to tease out integrability properties of the time derivative of $\ln(n + 1)$ and for some additional compactness properties). It in fact allows us in conjunction with some functional inequalities derived from the Trudinger--Moser inequality seen in \Cref{lemma:functional_ineq} to conclude that terms of the form
\[
	\int_\Omega (n + 1)\ln\left( \frac{n+1}{\overline{n} + 1} \right)
\]
and
\[
	\int_\Omega n \grad \phi \cdot u
\]
have similar time integrability properties. Time integrability of the former term proves useful to simplify some compactness results for $n$ that have already been shown in \cite{WinklerStokesCase}, while integrability of the latter term will be the central keystone to showing sufficient $L^2$ type bounds for $u$ and $\grad u$ (cf.\ \Cref{lemma:basic_u_bounds}) by testing the third equation with $u$ itself. While not as strong a set of bounds for $u$ as in the Stokes case, this then proves to be enough for the construction of global mass-preserving generalized solutions in the sense of \Cref{definition:weak_solution} below via a similar approximation approach as the one seen in \cite{WinklerStokesCase}.
 
\section{Generalized solution concept and approximate solutions}
As regularity information is rather hard to come by due to the complications outlined in the introduction, we will not endeavor to construct a global classical solution for (\ref{problem}), but instead confine ourselves to a very generalized notion of solution.
Because of the similarities to the pure Stokes case in \cite{WinklerStokesCase} and our desire to not unnecessarily duplicate effort, we let ourselves be guided by the generalized solution concept seen in said reference, which we of course slightly adapt to the full Navier--Stokes case. This reads as follows:
\begin{definition}
	\label{definition:weak_solution}
	We call a triple of functions 
	\begin{align*}
		&n \in L^\infty((0,\infty); L^1(\Omega)), \\
		&c \in L^\infty_\loc(\overline{\Omega}\times[0,\infty)) \cap L^2_\loc([0,\infty); W^{1,2}(\Omega)) \;\; \text{ and } \;\; \numberthis\label{wsol:regularity} \\
		&u \in L^\infty_\loc([0,\infty);(L^2(\Omega))^2) \cap L^2_\loc([0,\infty);(W_0^{1,2
		}(\Omega))^2)		
	\end{align*}
	with $n \geq 0$, $c \geq 0, \div u = 0$ a.e.\ in $\Omega\times(0,\infty)$,
	\begin{equation}
		\int_{\Omega} n(\cdot,t) = \int_{\Omega} n_0 \;\;\;\; \text{ for a.e.\ } t > 0 \label{wsol:mass_perservation}
	\end{equation}
	and 
	\begin{equation}
		\ln(n + 1) \in L^2_\loc([0,\infty);W^{1,2}(\Omega))
		\label{wsol:ln_n_regularity}
	\end{equation}
	a global mass-preserving generalized solution of (\ref{problem}), (\ref{boundary_conditions}), (\ref{initial_data}) and (\ref{initial_data_props}) if the inequality
	\begin{align*}
		-\int_0^\infty \int_{\Omega} \ln(n+1)\varphi_t - \int_{\Omega} \ln(n_0 + 1)\varphi(\cdot, 0) \geq& \int_0^\infty \int_{\Omega} \ln(n+1)\laplace\varphi + \int_0^\infty \int_{\Omega} |\grad \ln(n+1)|^2\varphi \\
		& - \int_0^\infty \int_{\Omega} \frac{n}{n+1} \grad \ln(n+1) \cdot (S(x,n,c)\grad c)\varphi \\
		& + \int_0^\infty \int_{\Omega} \frac{n}{n+1} (S(x,n,c) \grad c)\cdot \grad \varphi \\
		& + \int_0^\infty \int_\Omega \ln(n+1)(u\cdot \grad \varphi) \numberthis \label{wsol:ln_n_inequality}
	\end{align*}
	holds for all nonnegative $\varphi\in C_0^\infty(\overline{\Omega}\times[0,\infty))$ with $\grad \varphi \cdot \nu = 0$ on $\partial \Omega\times[0,\infty)$, if further
	\begin{equation}
		\int_0^\infty\int_{\Omega} c\varphi_t + \int_{\Omega} c_0\varphi(0, \cdot) = \int_0^\infty \int_\Omega \grad c \cdot \grad \varphi + \int_0^\infty \int_\Omega n f(c) \varphi - \int_0^\infty \int_\Omega c(u \cdot \grad \varphi) \label{wsol:c_equality}
	\end{equation}
	holds for all $\varphi \in L^\infty(\Omega\times(0,\infty)) \cap L^2((0,\infty); W^{1,2}(\Omega))$ having compact support in $\overline{\Omega}\times[0,\infty)$ with $\varphi_t \in L^2(\Omega\times(0,\infty))$, and if finally 
	\begin{equation}
		-\int_0^\infty \int_\Omega u \cdot \varphi_t - \int_\Omega u_0 \cdot \varphi(\cdot, 0) = -\int_0^\infty \int_\Omega \grad u \cdot \grad \varphi + \int_0^\infty \int_\Omega (u \otimes u)\cdot\grad \varphi + \int_0^\infty\int_\Omega n\grad\phi \cdot \varphi
		\label{wsol:u_equality}
	\end{equation}
	holds for all $\varphi \in C_0^\infty(\Omega\times[0,\infty); \R^2)$ with $\div \varphi = 0$ on $\Omega\times[0,\infty)$.
\end{definition} 
\begin{remark}
It can be shown that generalized solutions of this type become classical if the functions $n$, $c$ and $u$ are sufficiently regular. The argument for this can be sketched as follows: For the $u$ and $c$ components of the solution, this is fairly obvious as both satisfy the standard variational formulation of both of their equations, while the $n$ component presents somewhat more of a challenge as it only satisfies a very specific integral inequality and mass conservation property. That this is already sufficient has been for instance argued in \cite[Lemma 2.1]{WinklerLargeDataGeneralized} for the case $u \equiv 0$ and the argument transfers easily. 
\\[0.5em]
Furthermore at least in the Stokes case, some similar generalized solutions have been shown to eventually (from some large $t > 0$ onwards) attain such a level of regularity (cf.\ \cite{WinklerEventualSmoothness}).
\end{remark}
\noindent Similar to Reference \cite{WinklerStokesCase}, the key to our existence results will be approximate problems defined in the following way: We first fix families of functions $(\rho_\varepsilon)_{\varepsilon \in (0,1)}$ and $(\chi_\varepsilon)_{\varepsilon \in (0,1)}$ with 
\[
	\rho_\varepsilon \in C_0^\infty(\Omega) \stext{such that} 0 \leq \rho_\varepsilon \leq 1 \text{ in } \Omega \stext{and} \rho_\varepsilon \nearrow 1 \text{ pointwise in } \Omega \text{ as } \varepsilon \searrow 0
\]
and 
\[
	\chi_\varepsilon \in C_0^\infty([0,\infty)) \stext{such that} 0 \leq \chi_\varepsilon \leq 1 \text{ in } [0,\infty) \stext{and} \chi_\varepsilon \nearrow 1 \text{ pointwise in } [0,\infty) \text{ as } \varepsilon \searrow 0
\]
constructed by standard methods.
For $\varepsilon \in (0,1)$, we then define
\[
	S_\varepsilon(x,n,c) \defs \rho_\varepsilon(x) \chi_\varepsilon(n) S(x,n,c), \;\;\;\;\;\; \forall (x,n,c) \in \overline{\Omega}\times[0,\infty)^2
\]
and consider the following initial boundary value problem:
\[
\left\{
	\begin{aligned}
		{n_\varepsilon}_t + u_\varepsilon \cdot \grad n_\varepsilon \;=&\;\; \laplace n_\varepsilon - \div (n_\varepsilon S_\varepsilon(x, n_\varepsilon, c_\varepsilon)\grad c_\varepsilon), \;\;\;\;\;\;\; && x \in \Omega, \; t > 0,\\
		{c_\varepsilon}_t + u_\varepsilon \cdot \grad c_\varepsilon \;=&\;\; \laplace c_\varepsilon - n_\varepsilon f(c_\varepsilon), && x \in \Omega, \; t > 0,\\
		{u_\varepsilon}_t + (u_\varepsilon \cdot \grad ) u_\varepsilon \;=&\;\; \laplace u_\varepsilon + \grad P_\varepsilon + n_\varepsilon\grad \phi,  && x \in \Omega, \; t > 0, \\
		\div u_\varepsilon \;=&\;\; 0,  && x \in \Omega, \; t > 0, \\
		\grad n_\varepsilon \cdot \nu = \grad c_\varepsilon \cdot \nu = 0&, \;\; u_\varepsilon = 0, && x \in \partial\Omega, \; t > 0, \\
		n_\varepsilon(x,0) = n_0(x), \;\; &c_\varepsilon(x,0) = c_0(x), \;\; u_\varepsilon(x,0) = u_0(x), && x \in \Omega.
	\end{aligned}
	\numberthis \label{approx_system}
\right. 
\]
This kind of regularized version of (\ref{problem}) with (\ref{boundary_conditions}) and (\ref{initial_data}) then easily admits a global classical solution because it substitutes standard Neumann boundary conditions for the more complex no-flux boundary condition seen in (\ref{boundary_conditions}) and makes the first equation accessible to comparison arguments with a non-zero constant to gain a global upper bound for $n_\varepsilon$, which would be much harder to achieve otherwise. This of course only works under similar assumptions on the parameter functions $f$, $S$, $\phi$ as proposed in the introduction. As the techniques to achieve an existence result for the approximated system above are fairly well-documented and do not appreciably differ for our case in comparison to e.g.\ the case studied in \cite{WinklerExistence}, we will only give a short sketch to justify the following existence theorem:
\begin{theorem}
	For $\varepsilon \in (0,1)$, initial data of regularity (\ref{initial_data_props}) and $f$, $S$, $\phi$ satisfying (\ref{f_regularity})--(\ref{phi_regularity}), there exist functions
	\begin{align*}
		n_\varepsilon, c_\varepsilon &\in C^0(\overline{\Omega}\times[0,\infty))\cap C^{2,1}(\overline{\Omega}\times(0,\infty)), \\
		u_\varepsilon &\in C^0(\overline{\Omega}\times[0,\infty);\R^2)\cap C^{2,1}(\overline{\Omega}\times(0,\infty);\R^2) \\
		P_\varepsilon &\in C^{1,0}(\Omega\times(0,\infty))
	\end{align*}
	such that $n_\varepsilon > 0$, $c_\varepsilon \geq 0$ on $\overline{\Omega}\times(0,\infty)$ and $(n_\varepsilon, c_\varepsilon, u_\varepsilon, P_\varepsilon)$ is a classical solution of (\ref{approx_system}).
	\label{lemma:approx_existence}
\end{theorem}
\begin{proof}
	Standard contraction mapping approaches in an appropriate setting (as e.g.\ seen in \cite[Lemma 2.1]{WinklerExistence} for a similar system) provide us with a classical solution for (\ref{approx_system}) on a space time cylinder $\Omega\times[0,T_{\text{max}, \varepsilon})$ with some maximal $T_{\text{max}, \varepsilon} \in (0,\infty]$ and a blow-up criterion of the following type:
	\[
		\text{If } T_{\text{max}, \varepsilon} < \infty, \text{ then } \limsup_{t\nearrow T_{\text{max}, \varepsilon}} \left( \; \|n_\varepsilon(\cdot, t)\|_{L^\infty(\Omega)} +\|c_\varepsilon(\cdot,t)\|_{W^{1,q}(\Omega)} + \|A_2^\alpha u_\varepsilon(\cdot,t)\|_{L^2(\Omega)} \; \right) = \infty.
	\]
	Here $q$ is some number greater than $2$. Non-negativity and positivity on $\overline{\Omega}\times(0,\infty)$ of $c_\varepsilon$ and $n_\varepsilon$ respectively are immediately ensured by maximum principle.
	Note now further that, because we defined the $S_\varepsilon$ to be zero for sufficiently large values of $n$, well-known comparison arguments can be used to already gain a global upper bound for $n_\varepsilon$ on the whole cylinder. This already rules out blowup regarding $n_\varepsilon$. As the second equation in (\ref{approx_system}) is generally fairly unproblematic, similar boundedness results can be achieved for $c_\varepsilon$ (cf.\ \Cref{lemma:basic_props} later in this paper). Regarding the possible blowup of $c_\varepsilon$ or $u_\varepsilon$, we can look at the prior work done in Section 4.2 of Reference \cite{WinklerExistence}, where it is proven for much weaker prerequisites as already established here that the two norms in the blowup criteria regarding $c_\varepsilon$ and $u_\varepsilon$ respectively are bounded as well. Note that this is mostly done using the second and third equation of the system studied in said reference, which are apart from some slight generalizations the same as the second and third equation in (\ref{approx_system}). Only one step in the reference uses an energy inequality not available to us to establish a bound for $\int_\Omega |\grad c_\varepsilon|^2$, which is also easily gained by a straightforward testing procedure for the second equation in (\ref{approx_system}) without using said energy inequality.
\end{proof}\noindent
For the rest of this paper, we now fix initial values $(n_0, c_0, u_0)$ of regularity class (\ref{initial_data_props}) and parameter functions $f$, $S$ and $\phi$ satisfying (\ref{f_regularity})--(\ref{phi_regularity}). We then further fix a corresponding family of solutions $(n_\varepsilon, c_\varepsilon, u_\varepsilon, P_\varepsilon)_{\varepsilon \in (0,1)}$ according to \Cref{lemma:approx_existence}.
\section{A priori estimates for the approximate solutions}
\subsection{Results for $n_\varepsilon$ and $c_\varepsilon$ reusable from the Stokes case}
Let us now start by revisiting some key results for the approximate solutions $n_\varepsilon, c_\varepsilon$, which can be derived in a very similar way to the Stokes case (cf.\ \cite{WinklerStokesCase}) as they stem from just considering the first two equations in (\ref{approx_system}). 
\begin{lemma}
\label{lemma:basic_props}
Let $\varepsilon \in (0,1)$. The mass conservation equality
\begin{equation}
	\int_\Omega n_\varepsilon(x,t) \d x  = \int_\Omega n_0 \label{eq:mass_perservation}
\end{equation}
holds for all $t > 0$ and, for each $p \in [1,\infty]$, the inequality
\begin{equation}
	\|c_\varepsilon(\cdot, t)\|_\L{p} \leq \|c_\varepsilon(\cdot, s)\|_\L{p}
	\label{eq:c_monoticity}
\end{equation}
holds for $t \geq s \geq 0$. We further have that 
\begin{equation}
	\int_0^\infty\int_\Omega |\grad c_\varepsilon|^2 \leq \frac{1}{2}\int_\Omega c_0^2
	\label{eq:grad_c_bound}
\end{equation}
and 
\begin{equation}
	\int_0^\infty\int_\Omega \frac{|\grad n_\varepsilon|^2}{(n_\varepsilon + 1)^2} \leq 2\int_\Omega n_0 + S_0^2(\|c_0\|_\L{\infty})\int_\Omega c_0^2 \label{eq:weak_grad_n_smallness}
\end{equation}
with $S_0$ as in (\ref{S_0_bound}).
\end{lemma}
\begin{proof}
The mass conservation property (\ref{eq:mass_perservation}) is immediately evident after integration of the first equation in (\ref{approx_system}). Further, testing the second equation in (\ref{approx_system}) with $c_\varepsilon^{p-1}$ for $p \in [1,\infty)$ gives us 
\[
	\frac{1}{p}\int_\Omega c^p_\varepsilon(\cdot, t) + (p-1)\int_0^t \int_\Omega c_\varepsilon^{p-2} |\grad c_\varepsilon|^2 + \int_0^t \int_\Omega n_\varepsilon c_\varepsilon^{p-1} f(c_\varepsilon) = \frac{1}{p} \int_\Omega c_0^p  \;\;\;\; \forall t > 0,
\]
which for fitting choices of $p$ yields (\ref{eq:c_monoticity}) for finite $p$ as well as (\ref{eq:grad_c_bound}). The case $p = \infty$ in (\ref{eq:c_monoticity}) then follows via the limit process $p \rightarrow \infty$. To now derive (\ref{eq:weak_grad_n_smallness}), we test the first equation in (\ref{approx_system}) with $\frac{1}{n_\varepsilon + 1}$ to obtain
\[
	\frac{\d }{\d t} \int_\Omega \ln(n_\varepsilon + 1) = \int_\Omega \frac{|\grad n_\varepsilon|^2}{(n_\varepsilon + 1)^2} - \int_\Omega \frac{n_\varepsilon}{(n_\varepsilon + 1)^2} \grad n_\varepsilon \cdot S_\varepsilon(x, n_\varepsilon, c_\varepsilon) \grad c_\varepsilon \;\;\;\; \forall t > 0,
\]
which we then further improve to
\[
	\frac{1}{2}\int_\Omega \frac{|\grad n_\varepsilon|^2}{(n_\varepsilon + 1)^2} \leq \frac{\d}{\d t} \int_\Omega \ln(n_\varepsilon + 1) + \frac{S_0^2(\|c_0\|_\L{\infty})}{2} \int_\Omega|\grad c_\varepsilon|^2 \;\;\;\; \forall t > 0
\]
by application of Young's inequality,  (\ref{S_0_bound}) and (\ref{eq:c_monoticity}). Time integration of the above in combination with (\ref{eq:grad_c_bound}) then results in the desired inequality (\ref{eq:weak_grad_n_smallness}).
\end{proof}
\subsection{Estimates for $u_\varepsilon$ based on the Trudinger--Moser inequality}
While semigroup methods proved very fruitful when the third equation in (\ref{approx_system}) is of Stokes type, they are in our case thoroughly thwarted by the nonlinear convection term $(u_\varepsilon \cdot \grad) u_\varepsilon$. As such, we will use very different tools to at least partially recover some of the $L^p$ boundedness results for $u_\varepsilon$ and its derivatives seen in Lemma 3.3 of Reference \cite{WinklerStokesCase} for the Stokes case. To fill the role of these tools, we therefore start by deriving two functional inequalities based on the Trudinger--Moser inequality (pioneered in \cite{trudinger1967imbeddings}, \cite{moser1971sharp}) and inspired by \cite{WinklerMoserTrudingerFluidInteraction}:
\begin{lemma}
	\label{lemma:functional_ineq}
	There exists a constant $C > 0$ such that
	\begin{equation}
		\int_\Omega \varphi (\psi - \overline{\psi}) \leq \frac{1}{a} \left[\;\int_\Omega \psi \ln\left(\frac{\;\psi\;}{\overline{\psi}}\right)  + C\int_\Omega \psi \; \right] + \frac{a}{8\pi}\left\{\int_\Omega \psi\right\} \int_\Omega |\grad \varphi|^2  \label{eq:fucntional_ineq_1}
	\end{equation}
    and
	\begin{equation}
		\int_\Omega \psi \ln\left(\frac{\;\psi\;}{\overline{\psi}}\right) \leq \frac{1}{2\pi} \left\{\int_\Omega \psi \right\}\int_\Omega \frac{|\grad \psi|^2}{\psi^2} + C\int_\Omega \psi
		\label{eq:fucntional_ineq_2}
	\end{equation}
	with $\overline{\psi} \defs \frac{1}{|\Omega|} \int_\Omega \psi$ for all $\varphi \in C^1(\overline{\Omega})$, positive $\psi \in C^1(\overline{\Omega})$ and $a > 0$.
\end{lemma}  
\begin{proof}
	By using the Trudinger--Moser inequality seen in Proposition 2.3 of Reference \cite{chang1988conformal}, which is applicable because $\Omega$ is convex and therefore finitely connected, we gain a constant $K_1 \geq |\Omega|$ with
	\begin{equation}
		\int_\Omega e^{2\pi\xi^2} \leq K_1 \label{eq:moser-trudinger-raw}
	\end{equation}
	for all $\xi\in C^1(\overline{\Omega})$ with $\int_\Omega \xi = 0$ and $\int_\Omega |\grad \xi|^2 \leq 1$. Note that we can choose the constant $\beta$ seen in Proposition 2.3 from the reference equal to $2\pi$ because $\Omega$ has a smooth boundary. Using Young's inequality to see that
	\[
		\xi - \overline{\xi} \leq |\xi - \overline{\xi}| \leq 2\pi\left(\frac{\xi - \overline{\xi}}{\|\grad \xi\|_\L{2}}\right)^2 + \frac{1}{8\pi} \int_\Omega |\grad \xi|^2
	\]
	with $\overline{\xi} \defs \frac{1}{|\Omega|} \int_\Omega \xi$, the inequality in (\ref{eq:moser-trudinger-raw}) directly implies that
	\[
		\int_\Omega e^{\xi-\overline{\xi}} \leq K_1\exp\left( \frac{1}{8\pi} \int_\Omega |\grad \xi|^2 \right)
	\]
	or further that
	\begin{equation}
		\int_\Omega e^{\xi} \leq K_1\exp\left( \frac{1}{8\pi} \int_\Omega |\grad \xi|^2 + \frac{1}{|\Omega|}\int_\Omega \xi \right) \label{eq:moser-trudinger-cons}
	\end{equation}
	for all $\xi \in C^1(\overline{\Omega})$. 
	\\[0.5em]
	Now fix $\varphi\in C^1(\overline{\Omega})$, positive $\psi \in C^1(\overline{\Omega})$ and $a > 0$. 
	We then observe that the estimate
	\[
		\ln\left( \int_\Omega e^{a\varphi} \right) = \ln\left( \int_\Omega e^{a\varphi} \frac{m}{\psi} \frac{\psi}{m} \right) \geq \frac{a}{m}\int_\Omega \varphi \psi - \frac{1}{m}\int_\Omega \psi\ln\left(\frac{\psi}{m}\right) 
	\]
	holds with $m \defs \int_\Omega \psi$ because of Jensen's inequality. If we now apply (\ref{eq:moser-trudinger-cons}) with $\xi \defs a \varphi$ to this and multiply by $\frac{m}{a}$, we get that
	\[
		\int_\Omega \varphi \psi \leq \frac{1}{a}\int_\Omega \psi\ln\left(\frac{\psi}{m}\right)  + \frac{am}{8\pi} \int_\Omega |\grad \varphi|^2 + \frac{m}{|\Omega|}  \int_\Omega \varphi + \frac{m}{a}\ln(K_1)
	\]
	or further that
	\[
		\int_\Omega \varphi (\psi - \overline{\psi}) \leq \frac{1}{a} \left[ \; \int_\Omega \psi \ln\left(\frac{\;\psi\;}{\overline{\psi}}\right)  + \ln\left(\frac{K_1}{|\Omega|}\right)\int_\Omega \psi \; \right] + \frac{a}{8\pi}\left\{\int_\Omega \psi\right\} \int_\Omega |\grad \varphi|^2 \numberthis \label{eq:proto_functional_eq1}
	\]
	after some rearranging. This gives us our first result. 
	\\[0.5em] 
	Now only fix a positive $\psi \in C^{1}(\overline{\Omega})$.
	We can then choose $\varphi \defs \ln\left(\frac{\;\psi\;}{\overline{\psi}}\right)$ and $a \defs 2$ in (\ref{eq:proto_functional_eq1}) to get that
	\[
		\int_\Omega \psi \ln\left(\frac{\;\psi\;}{\overline{\psi}}\right) \leq \frac{1}{2} \int_\Omega \psi \ln\left(\frac{\;\psi\;}{\overline{\psi}}\right) + \frac{1}{4\pi} \left\{\int_\Omega \psi\right\} \int_\Omega \frac{|\grad \psi|^2 }{\psi^2} + \frac{1}{2}\ln\left(\frac{K_1}{|\Omega|}\right)\int_\Omega \psi + \overline{\psi}\int_\Omega \ln\left(\frac{\;\psi\;}{\overline{\psi}}\right).
	\] 
	Because by Jensen's inequality $\int_\Omega \ln\Big(\tfrac{\,\psi\,}{\overline{\psi}}\Big) \leq 0$, this directly implies our second result.
\end{proof}\noindent
Employing the second functional inequality (\ref{eq:fucntional_ineq_2}), we can now use the fairly weak regularity information in (\ref{eq:weak_grad_n_smallness}) to derive the following preliminary integrability property for the family $(n_\varepsilon)_{\varepsilon \in (0,1)}$, which will not only prove useful to derive bounds for the family $(u_\varepsilon)_{\varepsilon \in (0,1)}$, but also to later simplify a compactness argument, which was already used in the Stokes case.
\begin{lemma}
	\label{lemma:nlogn_bound}
	For each $T > 0$, there exists $C(T) > 0$ such that
	\begin{equation}
		\int_0^T \int_\Omega (n_\varepsilon + 1)\ln\left(\frac{n_\varepsilon + 1}{\overline{n_0} + 1}\right) \leq C(T)
		\label{eq:nlogn_bound}
	\end{equation}
	for all $\varepsilon \in (0,1)$.
\end{lemma}
\begin{proof}
	Applying the functional inequality (\ref{eq:fucntional_ineq_2}) from \Cref{lemma:functional_ineq} to $\psi \defs n_\varepsilon + 1$ directly yields that
	\begin{align*}
		\int_\Omega (n_\varepsilon + 1) \ln\left( \frac{\;n_\varepsilon + 1\;}{\overline{n_0} + 1}\right) \leq \frac{1}{2\pi}\left\{  \int_\Omega (n_\varepsilon + 1) \right\} \int_\Omega \frac{|\grad n_\varepsilon|^2}{(n_\varepsilon + 1)^2} + K_1 \int_\Omega (n_\varepsilon + 1)
		 \;\;\;\; \forall t \in [0,T] \numberthis \label{eq:initial_nlnn}
	\end{align*} for 
	all $\varepsilon \in (0,1)$ and some constant $K_1 > 0$ given by the lemma. We further know from \Cref{lemma:basic_props} that there exists a constant $K_2 > 0$ such that
	\[
		\int_0^\infty \int_\Omega \frac{|\grad n_\varepsilon|^2}{(n_\varepsilon + 1)^2} \leq K_2
	\]
	for all $\varepsilon \in (0,1)$. After time integration of (\ref{eq:initial_nlnn}), this then gives us that
	\[
		\int_0^T 	\int_\Omega (n_\varepsilon + 1) \ln\left( \frac{\;n_\varepsilon + 1\;}{\overline{n_0} + 1}\right) \leq \left(\frac{K_2}{2\pi} + K_1T\right) \int_\Omega (n_0 + 1) 
	\]
	for all $\varepsilon \in (0,1)$, which completes the proof.
\end{proof}\noindent
This in combination with another application of our functional inequalities above then serves as the basis to prove the central result of this section, namely the $L^2$ type bounds for the functions $(u_\varepsilon)_{\varepsilon \in (0,1)}$ and their gradients, which replace the results gained via semigroup arguments in the Stokes case:
\begin{lemma}
	\label{lemma:basic_u_bounds}
	For each $T > 0$, there exists a $C(T) > 0$ such that
	\begin{equation}
		\|u_\varepsilon(\cdot, t)\|_\L{2} \leq C(T) \;\;\;\;\;\; \forall t \in [0,T]
		\label{eq:u_bound}
	\end{equation}
	and 
	\begin{equation}
		\int_0^T \int_\Omega |\grad u_\varepsilon|^2 \leq C(T) \label{eq:grad_u_bound}
	\end{equation}
	for all $\varepsilon \in (0,1)$.
\end{lemma} 
\begin{proof}
As our first step, we test the third equation in (\ref{approx_system}) with $u_\varepsilon$ itself to gain that
\[
	\frac{1}{2}\frac{\d}{\d t}\int_\Omega |u_\varepsilon|^2 = -\int_\Omega |\grad u_\varepsilon|^2 + \int_\Omega n_\varepsilon \grad \phi \cdot u_\varepsilon = -\int_\Omega |\grad u_\varepsilon|^2 + \int_\Omega \grad \phi \cdot u_\varepsilon (n_\varepsilon - \overline{n_\varepsilon}) \;\;\;\; \forall t\in[0,T] \text{ and }\varepsilon \in (0,1).
\]
We now apply (\ref{eq:fucntional_ineq_1}) from \Cref{lemma:functional_ineq} (with $\varphi \defs \grad \phi \cdot u_\varepsilon$ and $\psi \defs n_\varepsilon + 1$) to the rightmost term in the previous inequality to gain a constant $K_1 > 0$ such that
\begin{align*}
	\frac{1}{2}\frac{\d}{\d t}\int_\Omega |u_\varepsilon|^2
	&\leq -\int_\Omega |\grad u_\varepsilon|^2 \\
	&+ \frac{1}{a} \left[\int_\Omega (n_\varepsilon + 1) \ln\left( \frac{\;n_\varepsilon + 1\;}{\overline{n_0} + 1}\right)  + K_1\int_\Omega (n_0 + 1) \right] + \frac{a}{8\pi}\left\{\int_\Omega (n_0 + 1)\right\} \int_\Omega |\grad (\grad \phi \cdot u_\varepsilon)|^2  \label{eq:u_with_u_test}\numberthis
\end{align*}
for any $a > 0$ and each $t \in [0,T], \varepsilon \in (0,1)$. Further note that
\begin{align*}
	\int_{\Omega}|\grad (\grad \phi \cdot u_\varepsilon)|^2  &\leq  2\int_{\Omega}| \grad \phi |^2 |\grad u_\varepsilon|^2 + 2\int_{\Omega} |H_\phi|^2 |u_\varepsilon|^2 \\
	&\leq 2\|\grad \phi\|_\L{\infty}^2\int_{\Omega} |\grad u_\varepsilon|^2 + 2\|H_\phi\|_\L{\infty}^2 \int_{\Omega} |u_\varepsilon|^2 \\
	&\leq  K_2 \int_{\Omega}|\grad u_\varepsilon|^2  \;\;\;\; \forall t\in[0,T] \text{ and } \varepsilon \in (0,1)
\end{align*}
with $K_2 \defs 2\|\grad \phi\|_\L{\infty}^2 + 2\|H_\phi\|_\L{\infty}^2 C_p^2$. Here $H_\phi$ denotes the Hessian of $\phi$ and $C_p$ is the Poincaré constant for $\Omega$. If we now apply this to (\ref{eq:u_with_u_test}) and set 
\[
	a \defs K_3 \defs \frac{8\pi}{2K_2}\left\{\int_\Omega (n_0 + 1) \right\}^{-1}  \;\;\;\; \forall t\in[0,T],
\]
we gain that
\begin{equation*}
	\frac{1}{2}\frac{\d}{\d t}\int_\Omega |u_\varepsilon|^2 + \frac{1}{2}\int_\Omega |\grad u_\varepsilon|^2 \leq \frac{g_\varepsilon(t)}{2} 
\end{equation*}
or rather
\begin{equation}
	\frac{\d}{\d t}\int_\Omega |u_\varepsilon|^2 + \int_\Omega |\grad u_\varepsilon|^2 \leq g_\varepsilon(t) 
	\label{eq:u_with_u_test_mod}
\end{equation}
with
\[
g_\varepsilon(t) \defs  \frac{2}{K_3} \left[\int_\Omega (n_\varepsilon + 1) \ln\left( \frac{\;n_\varepsilon + 1\;}{\overline{n_0} + 1}\right)  + K_1\int_\Omega (n_0 + 1) \right] \geq 0 
\]
for all $t\in[0,T]$ and $\varepsilon \in (0,1)$.
Because of \Cref{lemma:nlogn_bound}, there further exists a constant $K_4(T) > 0$ such that
\[
	\int_0^T g_\varepsilon(t) \d t \leq \frac{2}{K_3}\left[  \; K_4(T) + K_1 T \int_\Omega (n_0 + 1) \; \right] =: K_5(T)
\] 
for all $\varepsilon \in (0,1)$. If we now integrate (\ref{eq:u_with_u_test_mod}), this  property of $g_\varepsilon$ then directly gives us that
\[
	\int_\Omega |u_\varepsilon(\cdot, t)|^2 + \int_0^t \int_\Omega |\grad u_\varepsilon|^2  \leq \int_\Omega |u_0|^2 + \int_0^t g_\varepsilon(t) \leq \int_\Omega |u_0|^2 + K_5(T)
\]
for all $t\in[0,T]$, which immediately implies (\ref{eq:u_bound}) and (\ref{eq:grad_u_bound}).
\end{proof}
\subsection{Construction of limit functions}
Having now established weaker (though still suitable) uniform bounds for $u_\varepsilon$ than those seen in the Stokes case, we will make the final preparations for the construction of limit function for our family of approximate solutions as $\varepsilon \searrow 0$. We do this by proving some additional, albeit fairly weak, boundedness results for the time derivatives of all three families $(n_\varepsilon)_{\varepsilon \in (0,1)}$, $(c_\varepsilon)_{\varepsilon \in (0,1)}$ and $(u_\varepsilon)_{\varepsilon \in (0,1)}$ as this will provide the last prerequisite for some key applications of the Aubin-Lions lemma (cf.\ \cite{TemamNavierStokes}):
\begin{lemma}
	For all $T > 0$, there exists a constant $C(T) > 0$ such that
	\begin{gather}
		\int_0^T \|\partial_t \ln(n_\varepsilon(\cdot, t) + 1)\|_{(W^{2,2}_0(\Omega))^\star} \d t \leq C(T),
		\label{eq:dual_space_n_bound} \\
		\int_0^T \|{c_\varepsilon}_t(\cdot, t)  \|_{(W^{2,2}_0(\Omega))^\star}\d t \leq C(T)	
		\label{eq:dual_space_c_bound}
	\end{gather}
	and
	\begin{equation}
		\int_0^T \|{u_\varepsilon}_t(\cdot, t) \|_{(W^{2,2}_{0,\sigma}(\Omega))^\star} \d t \leq C(T)
		\label{eq:dual_space_u_bound}
	\end{equation}
	for all $\varepsilon \in (0,1)$.
	\label{lemma:dual_space_bounds}
\end{lemma}
\begin{proof}
	We will only give this proof in full detail for (\ref{eq:dual_space_u_bound}) and then just provide a sketch for (\ref{eq:dual_space_n_bound}) and (\ref{eq:dual_space_c_bound}) as all three inequalities can be proven in quite similar a fashion and more detailed proofs for (\ref{eq:dual_space_n_bound}) and (\ref{eq:dual_space_c_bound}) can be found in \cite{WinklerStokesCase}.
	\\[0.5em]
	As such, we first fix a $\psi \in W^{2,2}_{0,\sigma}(\Omega)$ and then test the third equation in (\ref{approx_system}) with $\psi$ to gain that
	\[
	\int_\Omega {u_\varepsilon}_t \cdot \psi = -\int_\Omega (u_\varepsilon \cdot \grad) {u_\varepsilon} \cdot \psi -\int_\Omega \grad {u_\varepsilon} \grad \psi -  \int_\Omega P_\varepsilon (\div \psi) +  \int_\Omega n_\varepsilon \grad \phi \cdot \psi \;\;\;\;\;\; \forall t \in [0,T] \text{ and } \varepsilon \in (0,1)
	\]
	after some partial integration steps.
	This then leads to 
	\begin{align*}
	\left| \; \int_\Omega {u_\varepsilon}_t \cdot \psi \; \right| &\leq 
	\int_\Omega |u_\varepsilon| |\grad u_\varepsilon| |\psi| + \int_\Omega |\grad u_\varepsilon| |\grad \psi| + \int_\Omega n_\varepsilon |\grad \phi| |\psi| \\
	&\leq \|u_\varepsilon\|_\L{2} \|\grad u_\varepsilon\|_\L{2} \|\psi\|_\L{\infty} + \|\grad u_\varepsilon\|_\L{2} \|\grad \psi\|_\L{2} + \|\grad \phi\|_\L{\infty} \|\psi\|_\L{\infty} \int_\Omega n_\varepsilon \\
	&\leq \left\{\left(\|u_\varepsilon\|_\L{2} + 1\right)\|\grad u_\varepsilon\|_\L{2} + \|\grad \phi\|_\L{\infty} \int_\Omega n_0 \right\}  \left(\|\grad \psi\|_\L{2} + \|\psi\|_\L{\infty} \right)
	\end{align*}
	for all $t \in [0,T]$ and $\varepsilon \in (0,1)$ by using the Cauchy--Schwarz inequality and the fact that $\div \psi = 0$. 
	By employing Young's inequality, the fact that $W^{2,2}(\Omega)$ embeds continuously into $L^\infty(\Omega)$ and the inequality (\ref{eq:u_bound}) from \Cref{lemma:basic_u_bounds}, we see that there exist constants $K_1(T), K_2(T) > 0$ such that
	\[
	\int_0^T \|{u_\varepsilon}_t \|_{(W^{2,2}_{0,\sigma}(\Omega))^\star} \leq \int_0^T \left(K_1(T) \int_\Omega |\grad u_\varepsilon|^2 + K_2(T)\right)
	\]
	for all $\varepsilon \in (0,1)$. Another application of \Cref{lemma:basic_u_bounds} and specifically the inequality (\ref{eq:grad_u_bound}) therein then directly gives us the desired bound for the family $(u_\varepsilon)_{\varepsilon \in (0,1)}$. 
	\\[0.5em]
	By testing the first equation in (\ref{approx_system}) with $\frac{\psi}{n_\varepsilon + 1}$ and the second equation in (\ref{approx_system}) with $\psi$ for any $\psi \in W^{2,2}(\Omega)$, we gain that
	\[
		\left|\;\int_\Omega \partial_t \ln(n_\varepsilon + 1) \cdot \psi \; \right| \leq K_3(T) \left\{ 1 + \int_\Omega \frac{|\grad n_\varepsilon|^2}{(n_\varepsilon + 1)^2} + \int_\Omega |\grad c_\varepsilon|^2 + \int_\Omega |u_\varepsilon|^2 \right\} \left(\|\grad \psi\|_\L{2} + \|\psi\|_\L{\infty} \right)
	\]
	and
	\[
		\left|\;\int_\Omega {c_\varepsilon}_t \cdot \psi \; \right| \leq K_3(T) \left\{  1 + \int_\Omega |\grad c_\varepsilon|^2 + \int_\Omega |u_\varepsilon|^2 \right\} \left(\|\grad \psi\|_\L{2} + \|\psi\|_\L{\infty} \right)
	\]
	for all $t \in [0,T]$, $\varepsilon \in (0,1)$ and some $K_3(T) > 0$ by similar techniques as seen above or in the proof of  Lemma 3.4 in Reference \cite{WinklerStokesCase}. Combining these two inequalities with \Cref{lemma:basic_props} and \Cref{lemma:basic_u_bounds} then yields the remaining two bounds for the families $(n_\varepsilon)_{\varepsilon \in (0,1)}$, $(c_\varepsilon)_{\varepsilon \in (0,1)}$ and therefore this completes the proof.
\end{proof}
\noindent
This then allows us to use essentially three applications of the Aubin-Lions lemma (cf.\ \cite{TemamNavierStokes}) to prove the following sequence selection and convergence result:
\begin{lemma}
	\label{lemma:subsequence_extraction}
	There exists a sequence $(\varepsilon_j)_{j\in\N} \in (0,1)$ with $\varepsilon_j \searrow 0$ as $j\rightarrow\infty$ such that
	\begin{equation}
	\left\{
		\begin{aligned}
		&n_\varepsilon \rightarrow n && \text{a.e.\ in } \Omega \times (0,\infty), \\
		&\ln(n_\varepsilon + 1) \rightharpoonup \ln(n + 1) \;\;\;\;\;\;\;\;\;\;\;\;\;\; && \text{in } L^2_\loc([0,\infty);W^{1,2}(\Omega)), \\
		&c_\varepsilon \rightarrow c  && \text{in } L^2_\loc(\overline{\Omega}\times[0,\infty))  \text{ and a.e.\ in } \Omega \times (0,\infty), \\
		&c_\varepsilon(\cdot, t) \rightarrow c(\cdot, t) && \text{in } L^2(\Omega) \text{ for a.e.\ } t > 0, \\
		&c_\varepsilon  \rightharpoonup c && \text{in } L^2_\loc([0,\infty);W^{1,2}(\Omega)) , \\
		& u_\varepsilon \rightarrow u  && \text{in } (L^2_\loc(\overline{\Omega}\times[0,\infty)))^2\text{ and a.e.\ in } \Omega \times (0,\infty) \;\;\;\; \text{ and } \\
		& u_\varepsilon \rightharpoonup u && \text{in } L^2_\loc([0,\infty);W^{1,2}_{0,\sigma}(\Omega))
		\end{aligned}
	\right. 
	\label{eq:basic_convergence_props}
	\end{equation}
	as $\varepsilon = \varepsilon_j \searrow 0$ and a triple of limit functions $(n,c,u)$ defined on $\Omega\times(0,\infty)$ and satisfying $n,c \geq 0$ and $\div u = 0$ almost everywhere.
\end{lemma}
\begin{proof}
	By applying \Cref{lemma:basic_props} and \Cref{lemma:dual_space_bounds} and the Aubin-Lions lemma (cf.\ \cite{TemamNavierStokes}), we immediately gain relative compactness of the families $(\ln(n_\varepsilon + 1))_{\varepsilon \in (0,1)}$ and $(c_\varepsilon)_{\varepsilon \in (0,1)}$ in $L^2_\loc([0,\infty);W^{1,2}(\Omega))$ with respect to the weak topology and in $L^2_\loc([0,\infty);L^2(\Omega))$ and therefore $L^2_\loc(\overline{\Omega}\times[0,\infty))$ with respect to the strong topology. Moreover, by the boundedness properties presented in \Cref{lemma:basic_u_bounds} and \Cref{lemma:dual_space_bounds} and another application of the Aubin-Lions lemma to the triple of function spaces
	\[
		W^{1,2}_{0,\sigma}(\Omega) \subseteq L^2_\sigma(\Omega) \subseteq (W^{2,2}_{0,\sigma}(\Omega))^\star,
	\]
	we gain relative compactness of the family $(u_\varepsilon)_{\varepsilon \in (0,1)}$ in $(L^2_\loc(\overline{\Omega}\times[0,\infty)))^2$ with respect to the strong topology and in $L^2_\loc([0,\infty);W^{1,2}_{0,\sigma}(\Omega))$ with respect to the weak topology because of inequalities (\ref{eq:u_bound}), (\ref{eq:grad_u_bound}) and (\ref{eq:dual_space_u_bound}). By multiple standard subsequence extraction arguments, we can therefore construct a sequence $(\varepsilon_j)_{j\in\N}$ such that $\varepsilon_j \searrow 0$ as $j \rightarrow \infty$ with the convergence properties seen in (\ref{eq:basic_convergence_props}).
	\\[0.5em]
	Note that the non-negativity properties directly transfer from the approximate functions because of the pointwise convergence, while $\div u = 0$ is directly ensured by $u$ being an element of $L^2_\loc([0,\infty);W^{1,2}_{0,\sigma}(\Omega))$.
\end{proof}

\subsection{Stronger convergence properties for $(n_{\varepsilon_j})_{j\in\N}$ and $(c_{\varepsilon_j})_{j\in\N}$}
While the convergence properties outlined in \Cref{lemma:subsequence_extraction} would already take us quite far in proving that the limit functions found in said lemma are in fact solutions in the sense of \Cref{definition:weak_solution}, we will still need to prove that the sequences converge in slightly stronger ways. Our first target for this is the sequence $(n_{\varepsilon_j})_{j\in\N}$ as it thus far exhibits the weakest convergence properties. We therefore want to show now that $n_{\varepsilon_j}$ converges towards $n$ in at least some $L^1$ fashion. While we can find a proof for this in Reference \cite{WinklerStokesCase}, which should still work in our case without any modification, we want to present a somewhat shorter argument here based on our \Cref{lemma:nlogn_bound}.
\begin{lemma}
	\label{lemma:l1_convergence}
	Let the function $n$ and the sequence $(\varepsilon_j)_{j \in \N}$ be as in \Cref{lemma:subsequence_extraction}. Then 
	\[
		n_\varepsilon \rightarrow n \;\;\;\; \text{ in } L^1_\loc(\overline{\Omega}\times[0,\infty)) \text{ as }\varepsilon = \varepsilon_j \searrow 0.
	\]
\end{lemma}
\begin{proof}
	Fix a $T > 0$. For
	\[
		G(t) \defs (t+1)\ln\left(\frac{t+1}{\overline{n_0} + 1}\right) \;\;\;\; \forall t \geq 0,
	\]
	observe that
	\[
		\lim_{t \rightarrow \infty} \frac{G(t)}{t} \geq \lim_{t \rightarrow \infty}\ln\left(\frac{t+1}{\overline{n_0} + 1}\right) = \infty
	\]
	and there exists a constant $C(T) > 0$ such that
	\[
		\int_0^T \int_\Omega G(|n_\varepsilon|) = \int_0^T \int_\Omega (n_\varepsilon + 1)\ln\left(\frac{n_\varepsilon + 1}{\overline{n_0} + 1}\right) \leq C(T)
	\]
	for all $\varepsilon \in (0,1)$ by \Cref{lemma:nlogn_bound}. Therefore, the family $(n_\varepsilon)_{\varepsilon \in (0,1)}$ fulfills the De La Vallée Poussin criterion for uniform integrability (cf.\ \cite[p.24]{PropabilitiesAndPotential}).
	\\[0.5em]
	Because \Cref{lemma:subsequence_extraction} furthermore ensures pointwise convergence of the sequence $n_{\varepsilon_j}$ almost everywhere, all prerequisites for the Vitali convergence theorem (cf.\ \cite[p.23]{PropabilitiesAndPotential}) are therefore met and it directly provides us with the desired $L^1_\loc(\overline{\Omega}\times[0,\infty))$ convergence as $T > 0$ was arbitrary. 
\end{proof}\noindent
Though we have already ensured certain weak convergence properties for the first derivatives of the sequence $(c_{\varepsilon_j})_{j\in\N}$, they appear to be insufficient to handle the chemotaxis derived terms in (\ref{wsol:ln_n_inequality}) of \Cref{definition:weak_solution} when passing to the limit to show that $(n,c,u)$ is in fact a generalized solution in the sense of said definition. As the techniques used to show the following stronger convergence property for $(\grad c_{\varepsilon_j})_{j\in\N}$ are very similar to the ones seen in the proof of Lemma 4.4 of Reference \cite{WinklerStokesCase} and Lemma 8.2 of Reference \cite{WinklerLargeDataGeneralized}, we will only give the following sketch:
\begin{lemma}
	\label{lemma:grad_u_convergence}
	Let the function $c$ and the sequence $(\varepsilon_j)_{j \in \N}$ be as in \Cref{lemma:subsequence_extraction}.
	Then $c$ satisfies (\ref{wsol:c_equality}) for all $\varphi$ as in \Cref{definition:weak_solution} and 
	\[
		\grad c_{\varepsilon} \rightarrow \grad c \;\;\;\; \text{ in } L^2_\loc(\overline{\Omega}\times[0,\infty)) \text{ as }\varepsilon = \varepsilon_j \searrow 0.
	\]
\end{lemma}
\begin{proof}
	Note first that the convergence properties seen in \Cref{lemma:subsequence_extraction} are already sufficient to show that $c$ in fact satisfies (\ref{wsol:c_equality}) for appropriate $\varphi$ by showing that all $c_\varepsilon$ do this and then passing to the limit. We will therefore not further expand on this as it is a fairly straightforward limit process. Nonetheless this property of $c$ is in fact the key to proving the desired convergence property for the gradients $\grad c_\varepsilon$.
	\\[0.5em]
	We first note that the convergence properties for $c$ shown in \Cref{lemma:subsequence_extraction} already give us that
	\[
		\int_0^T\int_\Omega |\grad c|^2 \leq \liminf_{\varepsilon=\varepsilon_j \searrow 0} \int_0^T\int_\Omega |\grad c_\varepsilon|^2
	\]
	for all $T > 0$ and we therefore only need to show that a similar estimate from below also holds.
	\\[0.5em]
	This is essentially a two step process: Firstly, we construct a family of test functions $\varphi$ that are essentially time averaged versions of $c$ with approximated initial data, which sufficiently approximate $c$ itself. This ensures the necessary regularity for the family of test functions to be used in (\ref{wsol:c_equality}) and after some limit processes we gain that
	\begin{equation}
		\frac{1}{2} \int_\Omega c^2(\cdot, T) - \frac{1}{2} \int_\Omega c_0^2 + \int_0^T \int_\Omega |\grad c|^2 \geq - \int_0^T \int_\Omega ncf(c)
		\label{eq:c_lower_bound_prep}
	\end{equation}
	for all $T \in (0,\infty)\setminus N$ whereby $N$ is a null set such that
	$(0,\infty)\setminus N$ contains only Lebesgue points of the map
	\[
		(0,\infty) \rightarrow [0,\infty), \;\; t \mapsto \int_\Omega c^2(x,t) \d x.
	\]
	Only considering these Lebsgue points is needed to ensure that some of the time averages converge properly. For all the details, see e.g.\ Lemma 8.1 in Reference \cite{WinklerLargeDataGeneralized}. 
\\[0.5em]
	Secondly by potentially enlarging $N$, we can further assume that outside of $N$ the integral $\int_\Omega c_\varepsilon(\cdot, t)^2$ converges to $\int_\Omega c(\cdot, t)^2$ as $\varepsilon = \varepsilon_j \searrow 0$ because of \Cref{lemma:subsequence_extraction} without loss of generality.
	Combining this with (\ref{eq:c_lower_bound_prep}), uniform $L^\infty$ boundedness of the family $(c_\varepsilon)_{\varepsilon \in (0,1)}$ from \Cref{lemma:basic_props} and the $L^1$ convergence property from \Cref{lemma:l1_convergence} then yields that
	\begin{align*}
		\int_0^T\int_\Omega |\grad c|^2 &\geq -\frac{1}{2} \int_\Omega c^2(\cdot, T) + \frac{1}{2} \int_\Omega c_0^2 - \int_0^T \int_\Omega ncf(c) \\
		&=\lim_{\varepsilon=\varepsilon_j \searrow 0} \left\{ -\frac{1}{2} \int_\Omega c_\varepsilon^2(\cdot, T) + \frac{1}{2} \int_\Omega c_0^2 - \int_0^T \int_\Omega n_\varepsilon c_\varepsilon f(c_\varepsilon) \right\} =\lim_{\varepsilon=\varepsilon_j \searrow 0} \int_0^T \int_\Omega |\grad c_\varepsilon|^2
	\end{align*}
	for all $T \in (0,\infty)\setminus N$. This is exactly the needed lower bound and therefore completes the proof.
\end{proof}

\section{Proof of Theorem 1.1}
Having now prepared all the necessary tools, we will prove the central theorem of this paper:
\begin{proof}[Proof of Theorem 1.1]
	We use the triple of limit functions $(n,c,u)$ found in \Cref{lemma:subsequence_extraction}
	as candidates for our generalized solution in the sense of \Cref{definition:weak_solution}. We have already established some of the properties needed for \Cref{definition:weak_solution} in \Cref{lemma:subsequence_extraction}, \Cref{lemma:grad_u_convergence} and therefore now only need to show that (\ref{wsol:regularity}),
	(\ref{wsol:mass_perservation}), (\ref{wsol:ln_n_inequality}) and (\ref{wsol:u_equality}) also hold.
	The mass preservation property (\ref{wsol:mass_perservation}) follows directly from our $L^1$ convergence result in \Cref{lemma:l1_convergence} and the mass preservation property of the approximate solutions seen in \Cref{lemma:basic_props}. 
	This together with  \Cref{lemma:l1_convergence} then further ensures that $n$ is of the appropriate regularity for (\ref{wsol:regularity}), while the remaining regularity properties for $c$ and $u$ are provided by the convergence properties in \Cref{lemma:subsequence_extraction}, the uniform $L^\infty$ bound for the sequence $(c_{\varepsilon_j})_{j\in\N}$ in \Cref{lemma:basic_props} and the uniform $L^2$ bound for the sequence $(u_{\varepsilon_j})_{j\in\N}$ in \Cref{lemma:basic_u_bounds}.
	\\[0.5em]
	It is further easy to see that the approximate solutions satisfy both (\ref{wsol:ln_n_inequality}) and
	(\ref{wsol:u_equality}) by partial integration and use of the boundary conditions in (\ref{approx_system}). We therefore only need to further argue that these properties survive the necessary limit process. For most terms in the integral equality (\ref{wsol:u_equality}) concerning $u$, this is fairly straightforward to show using the convergence properties established in \Cref{lemma:subsequence_extraction}, but we nonetheless give the full argument for at least the newly introduced term (compared to the Stokes case) as an example. This is especially pertinent as we needed to establish stronger convergence properties for $u$ to handle this term compared to \cite{WinklerStokesCase}, namely strong $L^2$ as opposed to weak $L^2$ convergence. 
	\\[0.5em]
	We first fix a $\varphi \in C^\infty_0(\Omega\times[0,\infty);\R^2)$. Then there exists $T > 0$ such that
	\[
	\text{supp}(\varphi) \subseteq \Omega\times[0,T].
	\] 
	We now observe that
	\begin{align*}
	&\left|\,\int_0^\infty\int_\Omega (u_\varepsilon \otimes u_\varepsilon)\cdot \grad \varphi - \int_0^\infty\int_\Omega (u \otimes u)\cdot \grad \varphi \,\right| \\
	&\leq \|\grad \varphi\|_{L^{\infty}(\Omega)}\left[\;\int_0^T\int_\Omega |u_\varepsilon - u||u_\varepsilon| + \int_0^T\int_\Omega |u| |u_\varepsilon - u| \;\right] \\
	&\leq \|\grad \varphi\|_{L^{\infty}(\Omega)}\left[ \;\|u_\varepsilon - u\|_{L^{2}(\Omega\times(0,T))}\|u_\varepsilon\|_{L^{2}(\Omega\times(0,T))} + \|u\|_{L^{2}(\Omega\times(0,T))} \|u_\varepsilon - u\|_{L^{2}(\Omega\times(0,T))}\; \right]
	\end{align*}
	for all $\varepsilon \in (0,1)$, which ensures that
	\begin{equation*}
	\int_0^\infty\int_\Omega (u_\varepsilon \otimes u_\varepsilon)\cdot \grad \varphi \rightarrow  \int_0^\infty\int_\Omega (u \otimes u) \cdot \grad \varphi \stext{ as } \varepsilon = \varepsilon_j \searrow 0
	\end{equation*}
	because of the $L_\loc^2(\Omega\times[0,\infty))$ convergence of the sequence $(u_{\varepsilon_j})_{j\in\N}$ towards $u$. 
	\\[0.5em]
    As it has been pretty thoroughly discussed in the proof of Theorem 1.1 of Reference \cite{WinklerStokesCase} that (\ref{wsol:ln_n_inequality}) similarly survives a corresponding limit process given similar convergence properties as proven here, we will not go into further depth regarding this point and refer the reader to the reference.
\end{proof}
\section*{Acknowledgement} The author acknowledges support of the \emph{Deutsche Forschungsgemeinschaft} in the context of the project \emph{Emergence of structures and advantages in cross-diffusion systems}, project number 411007140.

\end{document}